\documentclass{amsart}

\usepackage{amsmath, amssymb, amsbsy}
\usepackage{array}
\usepackage{graphpap, color, paralist, pstricks}
\usepackage[mathscr]{eucal}
\usepackage[pdftex]{graphicx}
\usepackage[pdftex,colorlinks,backref=page,citecolor=blue]{hyperref}
\usepackage{pifont}

\usepackage{mathtools}
\usepackage{cleveref}
\usepackage{float}
\usepackage{crossreftools}

\newcounter{bullet}

\usepackage{setspace}
\setdisplayskipstretch{1}

\usepackage{geometry}
\usepackage{comment}
\newtheorem{thm}{Theorem}[section]

\newtheorem{lem}[thm]{Lemma}

\newtheorem{conj}[thm]{Conjecture}

\theoremstyle{definition}

\newtheorem{claim}[thm]{Claim}

\newtheorem{obs}[thm]{Observation}

\crefname{lem}{lemma}{lemmas}
\crefname{thm}{theorem}{theorems}

\newcommand{\gl}{\lambda}

\newcommand{\RR}{\mathbb{R}}

\newcommand{\cF}{\mathcal{F} }
\newcommand{\cG}{\mathcal{G} }

\newcommand{\beq}[1]{\begin{equation}\label{#1}}
\newcommand{\enq}[0]{\end{equation}}

\newcommand{\nin}[0]{\noindent}

\newcommand{\sub}[0]{\subseteq}

\newcommand{\0}[0]{\emptyset}
\newcommand{\ra}[0]{\rightarrow}

\newcommand{\pr}[0]{\mathbb{P}}
\newcommand{\E}[0]{\mathbb{E}}

\renewcommand{\eta}{\left(\left(\frac{q}{2}\right)^2\right)}

\begin{document}

\title{Note on a conjecture of Talagrand: expectation thresholds vs. fractional expectation thresholds}

\author[Q. Dubroff]{Quentin Dubroff}
\address{Department of Mathematics, Carnegie Mellon University}
\email{qdubroff@andrew.cmu.edu}

\author[J. Kahn]{Jeff Kahn}
\address{Department of Mathematics, Rutgers University}
\email{jkahn@math.rutgers.edu}

\author[J. Park]{Jinyoung Park}
\address{Department of Mathematics, Courant Institute of Mathematical Sciences, New York University}
\email{jinyoungpark@nyu.edu}

\begin{abstract}
We show that a restricted version 
of a conjecture of M.\ Talagrand on the relation between  
``expectation thresholds" and ``fractional expectation thresholds" follows easily 
from a strong version of a second conjecture of Talagrand, on "selector processes." 
The selector process conjecture was proved by Park and Pham, and the  
quantitative strengthening used here is due to Bednorz, Martynek, and Meller. 
\end{abstract}

\maketitle

\section{Introduction}

The purpose of this note is to point out that a restricted version 
of a conjecture of 
M.\ Talagrand (\Cref{LT} below) is an easy consequence of 
\cite{bednorz2022suprema}, a quantitative strengthening of \cite[Theorem 1.5]{park2024conjecture}. 
We first briefly recall definitions. 
(For a less hurried introduction to the problem, see e.g. \cite{demarco2015note,frankston2022problem}.)

For a finite set $V$, $2^V$ is the power set of $V$. An $\cF \sub 2^V$ is \textit{increasing} if $B \supseteq A \in \cF \Rightarrow B \in \cF.$ For $\cG \sub 2^V$ we use $\langle \cG \rangle$ for the increasing family generated by $\cG$, namely $\{B \sub V: \exists A \in \cG, B \supseteq A\}.$ 
\emph{We assume throughout that $|V|=n$ and $\cF \sub 2^V$ is increasing.}

Say $\cF$ is \textit{$p$-small} if there is a $\cG \sub 2^V$ such that
\[\langle \cG \rangle \supseteq \cF \text{ and } \sum_{S \in \cG} p^{|S|} \le 1/2,\]
and set $q(\cF)=\max\{p:\text{$\cF$ is $p$-small}\}$. Say $\cF$ is \textit{weakly $p$-small} if there is a $\gl:2^V \ra \RR^+ (:= [0, \infty))$ such that
\beq{wps}
\sum_{S \sub I} \gl_S \ge 1 \,\, \forall I \in \cF \,\,
\text{ and } \,\,\sum_S \gl_S p^{|S|} \le 1/2,
\enq
and set $q_f(\cF)=\max\{p:\text{$\cF$ is weakly $p$-small}\}$. If there is 
$\gl$ as in \eqref{wps} supported on sets of size at most $r$, we say $\cF$ is 
\textit{weakly $(p,r)$-small}. 

We are interested in the following conjecture of Talagrand \cite[Conjecture 6.3]{talagrand2010many}.

\begin{conj}\label{LT}
    There is a universal $L>0$ such that for every finite $V$ and increasing $\cF \sub 2^V$,
    \[q(\mathcal F) \ge q_f(\mathcal F)/L.\]
\end{conj}

\nin 
That is, weakly $p$-small implies $(p/L)$-small.

Our goal here is the following statement, which says that for any $r$,
weakly $(p,r)$-small does imply $(p/L)$-small, but with $L$ depending on $r$.
The case $r=2$, which was proved in \cite{frankston2022problem}, was one of a pair of test cases for \Cref{LT} suggested in \cite{talagrand2010many}. (The other was resolved in \cite{demarco2015note}.) 
The simplicity of the present argument versus the rather difficult 
\cite{frankston2022problem} is a nice illustration of the power of the selector
process results we are using.

\begin{thm}\label{MTez}
    For any $r$, there is a $J$ such that any
    weakly $(Jp,r)$-small $\cF$ is $p$-small.
\end{thm}

This turns out to be an easy consequence of \Cref{BMM} below, which is essentially \cite[Lemma 3.3]{bednorz2022suprema}. For the statement of the lemma, we need a little more preparation. Let $\cF \sub 2^{V}$ be given, and for each $I \in \cF$, suppose we are given weights $(\mu_I(v))_{v \in I}$ with $\mu_I(v) \ge 0$ and
\[\mu_I(I):=\sum_{v \in I} \mu_I(v)=1.\]
Say $X \sub V$ is \textit{$c$-bad} ($c >0$) if
\[\sup_{I \in \mathcal F} \mu_I(X \cap I)<c.\]

\begin{lem}\label{BMM}
For each positive integer $r $ there is a $C$ such that
if $\cF $ is not $p$-small
and $W$ is chosen uniformly from the $m$-element subsets of $V$, then
\[\pr\left(\text{$W$ is $\left(1 - \frac{1}{2r}\right)$-bad}\right) < 2 \sum_{t=1}^n \left(C \frac{np}{m}\right)^t.\]
\end{lem}

\nin
When $r=1$ this is \cite[Lemma~3.3]{bednorz2022suprema}, and direct extension 
of the argument of \cite{bednorz2022suprema} gives Lemma~\ref{BMM} with 
$C = (2er)^2$.

\section{Proof of \Cref{MTez}}

We will prove Theorem~\ref{MTez} with $J= 10rC$, where $C=C(r)$ is as in Lemma~\ref{BMM}.\footnote{A more involved argument, using ideas from the proof of \cite[Theorem 8]{fischer2023some}, would allow us to take  $J=O(C)$, but for simplicity we settle for the present version.} 
For the rest of our discussion, $X$ is a uniformly random $m$-element subset of $V$, with $m = Jpn/(2r) = 5Cpn$.  

Suppose $\gl : \binom{V}{\leq r} \rightarrow \RR^+$ 
satisfies \eqref{wps} with $Jp$ in place of $p$.
We may assume $\gl_\0=0$, since otherwise $\gl'$ given by $\gl'_\0=0$ and 
$\gl'_S=\gl_S/(1-\gl_\0)$ if $S\neq \0$ also satisfies \eqref{wps} (with 1/2 improved to 
$(1-2\gl_\0)/(2(1-\gl_\0))$).
We use \Cref{BMM} in combination with the following easy point.

\begin{obs}\label{o1}
With probability at least $1/2$, $\sum_{S \sub X} \gl_S\le 1/(2r).$
\end{obs}

\begin{proof}
This follows from Markov's Inequality, once we notice that
our assumptions on $X$ and $\gl$ imply
    \[\E\left[\sum_{S\subseteq X} \gl_S\right] \leq  
    \sum_{S } (m/n)^{|S|} \gl_S 
    = \sum_{S } (Jp/(2r))^{|S|} \gl_S \leq (2r)^{-1}\sum_{S } (Jp)^{|S|} \gl_S \leq 1/(4r). \qedhere\]
\end{proof}

For $I \in \cF$, let $\nu_I = \sum_{S \sub I} |S|\gl_S \geq 1$
(see \eqref{wps}, recalling that $\gl_\0=0$).
Define $\mu_{I}:I \ra \RR^+$ by
\[\mu_{I}(v)=\nu_I^{-1}\sum_{v \in S \sub I} \gl_S  \,\,\,\, \forall v\in I\]
and notice that $\mu_I(I) = 1$. Suppose for a contradiction that $\cF$ is not $p$-small. Then by \Cref{BMM} and our choice of $m$, the probability that $X$ is $(1-1/(2r))$-bad (with respect to the $\mu_I$'s) is less than $1/2$. So the following claim gives the desired contradiction.

\begin{claim} With probability at least $1/2$,
$X$ is $(1 - 1/(2r))$-bad with respect to the $\mu_I$'s.
\end{claim}

\begin{proof}    For any $Y \sub V$ and $I \in \cF$,
\begin{align*}
\mu_{I}(Y \cap I) = \nu_I^{-1}\sum_{v \in Y \cap I}\sum_{v \in S \sub I}\gl_S 
&\leq \nu_I^{-1}\sum_{\substack{S \sub I \\ S\not\subseteq Y}} (|S| -1)\gl_S + \nu_I^{-1}\sum_{S \subseteq Y \cap I} |S|\gl_S \\
&= \nu_I^{-1}\sum_{ S \sub I}(|S| - 1)\gl_S + \nu_I^{-1}\sum_{S \subseteq Y\cap I} \gl_S\\
&\leq \nu_I^{-1}(1 - 1/r)\sum_{S \sub I}|S|\gl_S + \nu_I^{-1}\sum_{S \subseteq Y\cap I} \gl_S\\
&= 1-1/r+ \nu_I^{-1}\sum_{S \subseteq Y\cap I} \gl_S\le 1-1/r+
\sum_{S \subseteq Y\cap I} \gl_S.
\end{align*}
When $Y=X$ we may bound the last sum (for any $I$) by
$\sum_{S\sub X}\gl_S$, which
by \Cref{o1} is
at most $(2r)^{-1}$ with probability at least $1/2$; and it follows that
with probability at least $1/2$,
\[\sup_{I \in \cF} \mu_{I}(X \cap I) \leq 1 - 1/r + (2r)^{-1} = 1 - 1/(2r). \qedhere\]
\end{proof}

\section*{Acknowledgments}
JK was supported by NSF Grant DMS-1954035.  
JP was supported by NSF Grant DMS-2324978 and a Sloan Fellowship.

\bibliographystyle{plain}
\bibliography{bibliography}

\end{document}